\documentclass{amsproc}

\theoremstyle{plain} 

\newtheorem{theorem}{Theorem}[section]
\newtheorem{lemma}[theorem]{Lemma}
\newtheorem{remark}[theorem]{Remark}
\newtheorem{example}[theorem]{Example}
\newtheorem{corollary}[theorem]{Corollary}
\newtheorem{conjecture}[theorem]{Conjecture}
\newtheorem{definition}[theorem]{Definition}
\newtheorem{proposition}[theorem]{Proposition}
\newtheorem{question}[theorem]{Question}

\newcommand{\makeinvisible}[1]{}

\newcommand{\eop}{\ \hfill $\Box$}

\numberwithin{equation}{section}

\newcommand{\cc}{{\mathbb C}}
\newcommand{\pp}{{\mathbb P}}
\newcommand{\rr}{{\mathbb R}}

\newcommand{\ggg}{{\mathbb G}}
\newcommand{\hhh}{{\mathbb H}}
\newcommand{\aaa}{{\mathbb A}}
\newcommand{\Gm}{\ggg _m}

\newcommand{\Cc}{{\mathcal C}}

\newcommand{\Dd}{{\mathcal D}}

\newcommand{\Oo}{{\mathcal O}}
\newcommand{\Uu}{{\mathcal U}}

\newcommand{\Mm}{{\mathcal M}}

\newcommand{\rk}{{\rm rk}}

\begin{document}

\author[C. Simpson]{Carlos Simpson}
\address{CNRS, Laboratoire J. A. Dieudonn\'e, UMR 6621
\\ Universit\'e de Nice-Sophia Antipolis\\
06108 Nice, Cedex 2, France}
\email{carlos@unice.fr}
\urladdr{http://math.unice.fr/$\sim$carlos/} 
\thanks{This research is partially supported 
by ANR grants BLAN08-1-309225 (SEDIGA) and BLAN08-3-352054 (G-FIB)}

\title[Destabilizing modifications]{Iterated destabilizing modifications for vector bundles with connection}



\begin{abstract}
Given a vector bundle with integrable
connection $(V,\nabla )$ on a curve, if $V$ is not itself semistable as
a vector bundle then we can iterate a construction involving modification by
the destabilizing subobject to obtain a Hodge-like filtration $F^p$
which satisfies Griffiths transversality. The associated graded Higgs bundle
is the limit of $(V,t\nabla )$ under the de Rham to Dolbeault degeneration.
We get a stratification of the moduli space of connections, with as 
minimal stratum the space of opers. The strata have fibrations
whose fibers are Lagrangian subspaces of the moduli space. 
\end{abstract}

\maketitle

\section{Introduction} \label{sec-introduction}

Suppose $X$ is a smooth projective curve over $\cc$. Starting with a rank $r$ vector bundle with
integrable holomorphic connection $(V,\nabla )$,
if $V$ is semistable as a vector bundle, we get a point in the moduli space $\Uu (X,r)$ of semistable 
vector bundles of rank $r$ and degree $0$ on $X$. 

Let $M_{\rm DR}(X,r)$ denote the moduli space of vector bundles with integrable connection of rank $r$. 
The open subset $G_0$ where the underlying vector bundle is itself semistable
thus has a fibration 
$$
M_{\rm DR}(X,r)\supset G_0\rightarrow \Uu (X,r).
$$
The fiber over a point $[V] \in \Uu (X,r)$ (say a stable bundle) is the space of connections
on $V$, hence it is a principal homogeneous space on 
$H^0(End(V)\otimes \Omega ^1_X)\cong H^1(End (V))^{\ast} = T^{\ast}_V\Uu (X,r)$.
So, the above fibration is a twisted form of the cotangent bundle $T^{\ast}_V\Uu (X,r)\rightarrow \Uu (X,r)$.
At points where the bundle $V$ is not semistable, we will extend
$G_0$ to a stratification of $M_{\rm DR}(X,r)$ by locally closed subsets $G_{\alpha}$.

If  $V$ is not semistable, let
$H\subset V$ be the maximal destabilizing subsheaf. Recall that $H$ is a subsheaf whose
slope $\mu (H)=deg(H)/rk(H)$ is maximal, and among such subsheaves $H$ has maximal rank. It is unique, and is a strict subbundle
so the quotient $V/H$ is also a bundle.
 
The connection induces an algebraic map 
$$
\theta = \nabla : H \rightarrow (V/H) \otimes \Omega ^1_X.
$$
Define a Higgs bundle $(E,\theta )$ by setting $E^1:= H, E^0:= V/H$, $E:= E^0\oplus E^1$, and
$\theta$ is the above map. It is a ``system of Hodge bundles'', that is a Higgs bundle fixed by the
$\cc ^{\ast}$ action. 

If $(E,\theta )$ is a semistable Higgs bundle, the process stops. If not, we can continue by again choosing 
$H\subset (E,\theta )$ the maximal destabilizing sub-Higgs-bundle,  then using $H$ to further modify the filtration
according to the formula \eqref{modif} below. The proof of our first main Theorem \ref{convergence} consists in showing 
that this recursive process stops at a Griffiths-transverse filtration of
$(V,\nabla )$ such that the associated graded Higgs bundle is semistable. 

Classically filtered objects $(V,\nabla , F^{\cdot})$ arose from variations of Hodge structure. In case an irreducible connection
supports a VHS, our iterative procedure constructs the Hodge filtration $F^{\cdot}$ starting from
$(V,\nabla )$. 

In terms of the nonabelian Hodge filtration \cite{hfnac} $M_{\rm Hod}\rightarrow \aaa ^1$ the above process gives a way of calculating
the  limit point $\lim _{t\rightarrow 0}(t\lambda , V,t\nabla )$; the limit is a point in one of the connected components of
the fixed point set of the $\Gm$ action on the moduli space of Higgs bundles $M_{\rm H}$. Looking at where the limit lands gives
the stratification by $G_{\alpha}\subset M_{\rm DR}$. Existence of the limit is a generalization to $M_{\rm DR}$ of
properness of the Hitchin map for $M_{\rm H}$.
The interpretation in terms of Griffiths-transverse  filtrations was pointed out briefly in \cite{hfnac}. 

In the present paper, after describing the explicit and geometric construction of the limit point by iterating the destabilizing modification
construction, we consider various aspects of the resulting stratification. 
For example, we conjecture that the stratification is nested, i.e. smaller strata are contained in the closures of bigger ones.
This can be shown for bundles of rank $2$. 
A calculation in deformation theory shows that the set 
$L_q\subset M_{\rm DR}$ of points $(V,\nabla )$ such that $\lim _{t\rightarrow 0}(V,t\nabla ) = q$, is a lagrangian subspace for the natural symplectic form. 
We conjecture that these subspaces are closed and form a nice foliation \ref{foliation}. 
We mention in \S \ref{sec-parabolic} 
that the same theory works for the parabolic or orbifold cases, and point out a new phenomenon there: the biggest open {\em generic stratum}
no longer necessarily corresponds to unitary bundles. The Hodge type of the generic stratum varies with the choice of parabolic weights, with constancy
over polytopes. At the end of the paper we do some theoretical work (which was  missing from \cite{hfnac}) necessary for proving the existence
of limit points in the case of principal bundles. All along the way, we try to identify natural questions for further study.

\medskip

It is a great pleasure to dedicate this paper to Professor Ramanan. 
I would like to thank him for the numerous conversations we have had over the years,
starting from my time as a graduate student, in which he explained his insightful points of view on everything
connected to vector bundles. These ideas are infused throughout the paper. 

I would also like to thank Daniel Bertrand, Philip Boalch, David Dumas, Jaya Iyer, Ludmil Katzarkov, Bruno Klingler, Vladimir Kostov, 
Anatoly Libgober, Ania Otwinowska, Tony Pantev,
Claude Sabbah, and Szilard Szabo for interesting communications related
to the subjects of this paper.

\section{Griffiths transverse filtrations}
\label{gtfilt}
Suppose $X$ is  a smooth projective curve, and $V$ is a vector bundle with integrable 
holomorphic connection $\nabla : V\rightarrow V\otimes _{\Oo _X}\Omega ^1_X$. 
A {\em Griffiths transverse filtration} is a decreasing filtration of $V$ by strict subbundles 
$$
V = F^0 \supset F^1 \supset F^2 \cdots \supset F^k = 0
$$
which satisfies the {\em Griffiths transversality condition} 
$$
\nabla : F^p\rightarrow F^{p-1}\otimes _{\Oo  _X}\Omega ^1_X.
$$
In this case put $E := Gr_F(V):= \bigoplus _p E^p$ with $E^p:= F^p/F^{p+1}$. Define the $\Oo _X$-linear map
$$
\theta : E^p\rightarrow E^{p-1}\otimes _{\Oo _X}\Omega ^1_X
$$
using $\nabla$. Precisely, if $e$ is a section of $E^p$, lift it to a section $f$ of $F^p$ and 
note that by the transversality condition, $\nabla f$ is a section of 
$F^{p-1}\otimes _{\Oo _X}\Omega ^1_X$. Define $\theta (e)$
to be the projection of $\nabla (f)$ into $E^{p-1}\otimes _{\Oo _X}\Omega ^1_X$. 
Again by the transversality condition, $\theta (e)$ is independent of the choice of lifting $f$. If $a$ is a section of $\Oo _X$
then $\nabla (af) = a\nabla (f) + f\otimes da$ but the second term projects to zero in $E^{p-1}\otimes _{\Oo _X}\Omega ^1_X$,
so $\theta (a e) = a\theta (e)$, that is $\theta$ is $\Oo_X$-linear. 

We call $(E,\theta )$ the {\em associated-graded Higgs bundle} corresponding to $(V, \nabla , F^{\cdot})$. 

Griffiths-transverse filtrations are the first main piece of structure of variations of Hodge structure, and in that context the
map $\theta$ is known as the ``Kodaira-Spencer map''. 
This kind of filtration of a bundle with connection was generalized to the notion of ``good filtration'' for $\Dd$-modules,
and has appeared in many places. 

A complex variation of Hodge structure consists of a $(V,\nabla , F^{\cdot})$
such that furthermore there exists a $\nabla$-flat hermitian complex form which is nondegenerate on each piece of the filtration,
and with a certain alternating positivity property (if we split the filtration by an orthogonal decomposition, then the form should have
sign $(-1)^p$ on the piece splitting $F^p/F^{p+1}$). For a VHS, the associated-graded Higgs bundle
$(E,\theta )$ is semistable. 

The historical variation of Hodge structure picture is motivation for considering the filtrations and Kodaira-Spencer
maps, however we don't use the polarization which is not a complex holomorphic object. Instead, we concentrate on the semistability condition. 

\begin{definition}
\label{grsemistable}
We say that $(V, \nabla , F^{\cdot})$ is {\em gr-semistable} if the associated-graded Higgs bundle $(E,\theta )$ is semistable
as a Higgs bundle.
\end{definition}

The Higgs bundle $(E,\theta )$ is a fixed point of the $\cc ^{\ast}$ action, which is equivalent to saying that we have a 
structure of {\em system of Hodge bundles}, i.e. a decomposition
$E=\bigoplus _pE^p$ with $\theta : E^p\rightarrow E^{p-1}\otimes _{\Oo _X}\Omega ^1_X$. 

\begin{remark}
If $(E,\theta )$ is a system of Hodge bundles, then it is semistable as a Higgs bundle if and only if it is semistable as a system 
of Hodge bundles. In particular, if it is not a semistable Higgs bundle then the maximal destabilizing subobject $H\subset E$
is itself a system of Hodge bundles, that is $H=\bigoplus H^p$ with $H^p:= H\cap E^p$. 
\end{remark}

Indeed, if $(E,\theta )$ is not semistable, it is easy to see by uniqueness of the maximal destabilizing subsheaf that
$H$ must be preserved by the $\cc ^{\ast}$ action. 

These objects have appeared in geometric Langlands theory under the name ``opers'': 

\begin{example}
An {\em oper} is a vector bundle with integrable connection and Griffiths transverse filtration $(V,\nabla , F^{\cdot})$ 
such that $F^{\cdot}$ is a full flag, i.e. $E^p=Gr^p_F(V)$ are line bundles, and 
$$
\theta : E^p\stackrel{\cong}{\rightarrow} E^{p-1}\otimes _{\Oo _X}\Omega ^1_X
$$
are isomorphisms. 
\end{example}

If $g\geq 1$ then an oper is gr-semistable. This motivates the following definition and terminology. 

\begin{definition}
A {\em partial oper} is a vector bundle with integrable connection and Griffiths-transverse filtration $(V,\nabla , F^{\cdot})$
which is gr-semistable. 
\end{definition}

Every integrable connection supports at least one partial oper structure.

\begin{theorem}
\label{convergence}
Suppose $(V,\nabla )$ is a vector bundle with integrable connection on a smooth projective curve $X$. Then there exists
a gr-semistable Griffiths-transverse filtration giving a partial oper structure $(V,\nabla , F^{\cdot})$.
\end{theorem}

The proof will be given in \S \ref{constr} below. 

In general the filtration $F^{\cdot}$ in the previous theorem, is not unique, see Proposition \ref{filtunique}.
The associated-graded Higgs bundle
$(E=Gr^{\cdot}_F(V), \theta )$ is unique up to $S$-equivalence, as
follows from the interpretation in terms of the nonabelian Hodge filtration, \S \ref{nahf}.

Given the $S$-equivalence class, it makes sense to say whether $(E,\theta )$ is stable or not. 
Proposition \ref{filtunique} shows that the filtration $F^{\cdot}$ is unique up to shifting indices, if and only if
$(E,\theta )$ is stable.

In \S \ref{sec-parabolic}, the nonuniqueness of the filtration is related to a wall-crossing phenomenon in the parabolic case.

\section{Construction of a gr-semistable filtration}
\label{constr}

If $(E=\bigoplus E^p,\theta )$ is a system of Hodge bundles, for any $k$ let $E^{[k]}$ denote the system of Hodge bundles
with Hodge index shifted, so that 
$$
(E^{[k]})^p:= E^{p-k}.
$$

Let $(V,\nabla )$ be fixed. Suppose we are given a Griffiths-transverse filtration $F^{\cdot}$ such that $(Gr_F(V),\theta )$ is 
not a semistable Higgs bundle. Choose $H$ to be the maximal destabilizing subobject, which is a sub-system of Hodge bundles 
of $Gr_F(V)$. Thus 
$$
H= \bigoplus H^p, \;\;\; H^p\subset Gr^p_F(V)= F^pV/F^{p+1}V.
$$
Note that the $H^p$ are strict subbundles here. We can consider
$$
H^{p-1}\subset V/F^pV
$$
which is again a strict subbundle. 

Define a new filtration $G^{\cdot}$ of $V$ by 
\begin{equation}
\label{gdef}
G^p:= \ker \left( V \rightarrow \frac{V/F^pV}{H^{p-1}} \right) .
\end{equation}
The condition $\theta (H^p)\subset H^{p-1}\otimes _{\Oo _X}\Omega ^1_X$ means that the new filtration $G^{\cdot}$ is again
Griffiths-transverse. We have exact sequences
$$
0\rightarrow Gr^p_F(V)/H^p \rightarrow Gr^p_G(V)\rightarrow H^{p-1}\rightarrow 0
$$
which, added all together, can be written as an exact sequence of systems of Hodge bundles
\begin{equation}
\label{modif}
0\rightarrow Gr_F(V)/H \rightarrow Gr_G(V)\rightarrow H^{[1]}\rightarrow 0 .
\end{equation}

We would like to show that the process of starting with a filtration $F^{\cdot}$ and replacing it with the modified filtration $G^{\cdot}$
stops after a finite number of steps, at a gr-semistable filtration. As long as the result is still not gr-semistable, we can
choose a maximal destabilizing subobject and continue. To show that the process stops, we will define a collection of invariants which decrease in
lexicographic order. 

For a system of Hodge bundles $E$, let $\beta (E)$ denote the slope of the maximal destabilizing subobject. Let $\rho (E)$ denote
the rank of the maximal destabilizing subobject. Define the {\em center of gravity}  to be 
$$
\zeta (E):= \frac {\sum  \rk (E^p)\cdot p}{\rk (E)}.
$$
This measures the average location of the Hodge indexing. In particular, suppose $U=E^{[k]}$. Then $U^p= E^{p-k}$ so
$$
\zeta (U) = \frac {\sum p \rk (U^p)}{\rk (U)} = \frac {\sum p \rk (E^{p-k})}{\rk (E)} = \frac {\sum (p+k) \rk (E^p)}{\rk (E)} = \zeta (E)+k.
$$
This gives the formula
\begin{equation}
\label{zetaform}
\zeta (E^{[k]}) = \zeta (E) + k.
\end{equation}
Now for any non-semistable system of Hodge bundles $E$, let $H$ denote the maximal destabilizing subobject and put
$$
\gamma (E) := \zeta (E/H)-\zeta (H).
$$
This normalizes things so that $\gamma (E^{[k]}) = \gamma (E)$. 

Denote $\beta _F:= \beta (Gr_F(V),\theta )$, $\rho _F:= \rho (Gr_F(V),\theta )$, and $\gamma _F:= \gamma (Gr_F(V),\theta )$.

\begin{lemma}
\label{decreasing}
In the process $F^{\cdot}\mapsto G^{\cdot}$ described above, and assuming that $G^{\cdot}$ is also not gr-semistable,
then the triple of invariants $(\beta , \rho , \gamma )$ decreases strictly in the lexicographic ordering. 
In other words, $(\beta _G, \rho _G, \gamma _G)$ is strictly smaller than $(\beta _F, \rho _F, \gamma _F)$. 
\end{lemma}
\begin{proof}
Use the exact sequence \eqref{modif} and the formula \eqref{zetaform}. 
\end{proof}

In order to show that the $\gamma (E)$ remain bounded, observe the following.

\begin{lemma}
\label{nogap}
If $(V,\nabla )$ is an irreducible connection, and $F^{\cdot}$ is a Griffiths-transverse filtration,  then there are no gaps in
the $E^p=Gr^p_F(V)$, that is there are no $p' < p < p''$ such that $E^p=0$ but $E^{p'}\neq 0$ and $E^{p''}\neq 0$.
\end{lemma}
\begin{proof}
If there were such a gap, then by Griffiths transversality the piece $F^p=F^{p+1}$ would be a nontrivial subbundle preserved by
the connection, contradicting irreducibility of $(V,\nabla )$. 
\end{proof}

\begin{lemma}
\label{finite}
Suppose $(V,\nabla )$ is an irreducible connection. 
For all $E= Gr_ F(V)$ coming from  non-gr-semistable Griffiths-transverse filtrations of our fixed $(V,\nabla )$, 
each of the invariants $\beta (E)$, $\rho (E)$ and $\gamma (E)$ can take on only finitely many values.
\end{lemma}
\begin{proof}
For the slope and rank this is clear. For $\gamma (E)$ it follows from Lemma \ref{nogap}.
\end{proof}

\noindent
{\em Proof of Theorem \ref{convergence}}
Assume first of all that $(V,\nabla )$ is irreducible. 
Under the operation $F^{\cdot}\mapsto G^{\cdot}$, the triple of invariants $(\beta , \rho , \gamma )$ 
(which takes only finitely many values by Lemma 
\ref{finite})
decreases strictly in the lexicographic ordering by Lemma \ref{decreasing} until we get to a gr-semistable filtration. 

For a general $(V,\nabla )$, glue together the filtrations provided by the previous paragraph
on the semisimple subquotients of its Jordan-H\"older filtration. This can be done after possibly shifting the indexing of the
filtrations.
\eop

It is an interesting question to understand what would happen if we tried to do the above procedure in the case $\dim (X)\geq 2$
where the destabilizing subobjects could be torsion-free sheaves but not
reflexive.

\section{Interpretation in terms of the nonabelian Hodge filtration}
\label{nahf}

Consider the ``nonabelian Hodge filtration'' moduli space 
$$
M_{\rm Hod}(X,r) = \{ (\lambda , V, \nabla ),\;\; \nabla : V\rightarrow V\otimes \Omega ^1_X, \;\; \nabla (ae)= a\nabla (e)+\lambda d(a)e\}  
$$
with its map $\lambda : M_{\rm Hod}(X,r) \rightarrow \aaa ^1$, such that: 
\newline
---$\lambda ^{-1}(0) = M_H$ is the Hitchin moduli space of semistable Higgs bundles of rank $r$ and degree $0$; and 
\newline
---$\lambda ^{-1}(1) = M_{\rm DR}$ is the moduli space of integrable connections of rank $r$.

The group $\Gm$ acts on $M_{\rm Hod}$ over its action on $\aaa ^1$, via the formula $t\cdot (\lambda , V, \nabla ) = (t\lambda , V , t\nabla )$.
Therefore all of the fixed points have to lie over $\lambda = 0$, that is they are in $M_H$. We can write
$$
(M_H ) ^{\Gm } =  \bigcup _{\alpha} P_{\alpha}
$$
as a union of connected pieces. 

\begin{lemma}
\label{limitlem}
For any $y\in M_{\rm Hod}$, the limit $\lim _{t\rightarrow 0}t\cdot y$ exists, and is in one of the $P_{\alpha}$. 
\end{lemma}
\begin{proof}
An abstract proof was given in \cite{hfnac}. 
For $\lambda (y)\neq 0$
in which case we may assume $\lambda (y)=1$ i.e. $y\in M_{\rm DR}$, 
the convergence can also be viewed as a corollary of Theorem \ref{convergence}. Indeed, $y$ corresponds to a vector bundle with integrable connection
$(V,\nabla )$ and if we choose a gr-semistable Griffiths-transverse filtration $F^{\cdot}$ then the limit can be  calculated as
$$
\lim _{\lambda \rightarrow 0} (V,\lambda \nabla ) = (Gr_F(V),\theta ).
$$
This can be seen as follows. 
The Rees construction gives
a locally free sheaf 
$$
\xi (V,F) := \sum \lambda ^{-p}F^pV\otimes \Oo _{X\times \aaa ^1} \subset V\otimes \Oo _{X\times \Gm}.
$$
over $X\times \aaa ^1$ and by Griffiths transversality the product $\lambda \nabla$ extends to a $\lambda$-connection on $\xi (V,F)$ 
in the $X$-direction (here $\lambda$ denotes the coordinate on $\aaa ^1$). 
This family provides a morphism $\aaa ^1\rightarrow M_{\rm Hod}$ compatible with the $\Gm$-action and
having limit point 
$$
(\xi (V,F), \lambda \nabla )|_{\lambda = 0} = (Gr_F(V),\theta ).
$$
If $\lambda (y)=0$ i.e. $y\in M_H$ then a construction similar to that of Theorem \ref{convergence}
gives a calculation of the limit. Alternatively, on the moduli space of Higgs bundles 
it is easy to see from the properness of the Hitchin map \cite{Hitchin} that the limit exists. 
\end{proof}

We next note that the limit is unique. 
\begin{corollary}
\label{Sequivalent}
If $F^{\cdot}$ and $G^{\cdot}$ are two gr-semistable filtrations for the same $(V,\nabla )$ then
the Higgs bundles $(Gr _F(V),\theta _F)$ and $(Gr_G(V),\theta _G)$ are $S$-equivalent, that is the associated-graded polystable objects
corresponding to their Jordan-H\"older filtrations, are isomorphic. 
\end{corollary}
\begin{proof}
The moduli space is a separated scheme whose points correspond to $S$-equivalence classes of objects.
\end{proof}

Now, given that the limiting Higgs bundle is unique, we can use it to measure whether the partial oper structure will be unique or not:

\begin{proposition}
\label{filtunique}
Suppose $(V,\nabla )$ is a vector bundle with integrable connection and let $(E,\theta )$ be the unique polystable Higgs bundle 
in the $S$-equivalence class of the limit. Then the gr-semistable Griffiths transverse filtration for $(V,\nabla )$ is unique up to
translation of indices, if and only if $(E,\theta )$ is stable.
\end{proposition}
\begin{proof}
If the limiting Higgs bundle is stable, apply semicontinuity theory to the Rees families 
$(\xi (V,F), \lambda \nabla )$ and 
$(\xi (V,G), \lambda \nabla )$ on $X\times \aaa ^1$ (see \ref{limitlem}). The $Hom$ between these
is a rank one locally free sheaf over $\aaa ^1$ with action of $\Gm$, and this relative $Hom$ commutes 
with base change. After appropriately shifting one of
the filtrations we get a $\Gm$-invariant section which translates back to equality of the filtrations.

On the other hand, if the limiting Higgs bundle is not stable, we can choose a sub-system of Hodge bundles and
apply the construction \eqref{gdef} to change the filtration. The exact sequence \eqref{modif} shows that
the new filtration is different from the old.
\end{proof}

\section{The oper stratification}
\label{sec-os}

As is generally the case for a $\Gm$-action, the map $y\mapsto \lim _{t\rightarrow 0}t\cdot y$ is a constructible map from
$M_{\rm Hod}$ to the fixed point set $\bigcup _{\alpha}P_{\alpha}$. 

\begin{proposition}
\label{strata}
For any $\alpha$, the subset $G_{\alpha}\subset M_{\rm DR}(X,r)$ consisting of all points $y$ such that 
$\lim _{t\rightarrow 0}t\cdot y \in P_{\alpha}$ is locally closed. These partition the moduli space into 
the {\em oper stratification}
$$
M_{\rm DR}(X,r)=\bigcup _{\alpha}G_{\alpha}.
$$
Furthermore, for any point $p\in P_{\alpha}$ (which corresponds to an $S$-equivalence class of systems of Hodge bundles),
the set $L_p\subset M_{\rm DR}(X,r)$ of points $y$ with $\lim _{t\rightarrow 0}t\cdot y = p$, is a locally closed subscheme
(given its reduced subscheme structure). 
\end{proposition}
\begin{proof}
This is the classical Bialynicki-Birula theory: the moduli space can be embedded $\Gm$-equivariantly in $\pp ^N$ with a linear action;
the stratification and fibrations are induced by those of $\pp ^N$ but with refinement of the $P_{\alpha}$ into connected components. 
\end{proof}

For the moduli of Higgs bundles we have a similar stratification with strata denoted $\widetilde{G}_{\alpha}\subset M_{\rm H}(X,r)$ again defined
as the sets of points such that $\lim _{t\rightarrow 0}t\cdot y \in P_{\alpha}$.

\subsection{Opers}
\label{opers}

The {\em uniformizing Higgs bundles \cite{Hitchin}} are of the form $E=E^1\oplus E^0=L\oplus L'$ a direct sum
of two line bundles with $E^0=L'\cong L\otimes K^{-1}$, such that $\theta : E^1\rightarrow E^0\otimes \Omega ^1_X$ is an isomorphism.
The space of these is connected, determined by the choice of $L\in Pic ^{g-1}(X)$. For bundles of rank $r$ 
we can let $P_{\alpha}$ be the space of symmetric powers $Sym^{r-1}(E)$. 
These systems of Hodge bundles are rigid up to tensoring with a line bundle: the determinant map 
$P_{\alpha}\rightarrow Pic^0(X)$ is finite. 

The classical moduli space of $GL(r)$-opers \cite{BDopers} \cite{Frenkel} is the subset $G_{\alpha}$ defined in Proposition \ref{strata}
corresponding to the space $P_{\alpha}$ of symmetric powers of uniformizing Higgs bundles. 

The stratum of classical opers $G_{\alpha}$ is closed, because the corresponding stratum $\widetilde{G}_{\alpha}$ is
closed in $M_{\rm H}$. It also has minimal dimension among the strata, as can be seen from Lemma \ref{lagrange} below.
We conjecture that it is the unique closed stratum and the unique stratum of minimal dimension. These are easy to see
in the case of rank $2$, see \S \ref{nestedness}.

\subsection{Variations of Hodge structure}
\label{vhs}

If $(V,\nabla , F^{\cdot}, \langle \cdot , \cdot \rangle )$ is a polarized
complex VHS then the underlying filtration $F^{\cdot}$ (which is Griffiths-transverse by definition) is gr-semistable. 
The Higgs bundle $(Gr_F(V),\theta )$ is the one which corresponds to $(V,\nabla )$ by the
nonabelian Hodge correspondence. This implies that if $(V,\nabla )$ is a VHS with irreducible monodromy representation then it is gr-stable. 
In this case the filtration $F^{\cdot}$ is unique and the process of iterating the destabilizing modification 
described in \S \ref{constr} provides a construction
of the Hodge filtration starting from just the bundle with its connection $(V,\nabla )$.

For any stratum $G_{\alpha}$ as in Proposition \ref{strata},
let $G^{\rm VHS}_{\alpha}\subset M_{\rm DR}(X,r)$ be the real analytic  moduli space of polarized complex variations of Hodge structure
whose underlying filtered bundle is in $G_{\alpha}$. 
We have a diagram of real analytic varieties
$$
\begin{array}{rcl}
G^{\rm VHS}_{\alpha}\;\; & \stackrel{\rm real}{\hookrightarrow} & \;\;\;\; G_{\alpha} \\
 {\scriptstyle \cong}\searrow &  & \swarrow \\
&  P_{\alpha} &
\end{array} .
$$
Under the nonabelian Hodge identification $\nu : M_{\rm DR}(X,r) \cong M_{\rm H}(X,r)$
the space $G^{\rm VHS}_{\alpha}$ of variations of Hodge structure is equal to $P_{\alpha}$ and the diagonal isomorphism in the above diagram is the identity
when viewed in this way. 

The other points of $G_{\alpha}$ don't necessarily correspond to points of $\widetilde{G}_{\alpha}$
under the nonabelian Hodge identification $\nu$, and indeed it seems reasonable to conjecture that 
\begin{equation}
\label{conjGGtilde}
G^{\rm VHS}_{\alpha} = G_{\alpha}\cap \nu ^{-1}(\widetilde{G}_{\alpha}).
\end{equation}

In a similar vein, let $M_B(X,r)_{\rr}$ denote the real subspace of representations which go into some possibly indefinite unitary group $U(p,q)$. 
\begin{lemma}
\label{lemGMBR}
Restricting to the subset of smooth points,
$G^{\rm VHS}_{\alpha}$ is a connected component of $G_{\alpha} \cap M_B(X,r)_{\rr}$. 
\end{lemma}

The proof will be given in \S \ref{hodgedef} below. 

On the other hand, it is easy to see that there are 
other connected components too, for example when $p,q >0$ a general representation $\pi _1\rightarrow U(p,q)$ will still correspond to a stable
vector bundle, so it gives a point in $G_{0} \cap M_B(X,r)_{\rr}$ which is not a unitary representation. These points probably correspond to
twistor-like sections of Hitchin's twistor space, but which don't correspond to preferred sections. 

It is an interesting question to understand how to distinguish $G^{\rm VHS}_{\alpha}$ among the connected components of 
$G_{\alpha} \cap M_B(X,r)_{\rr}$. 

\subsection{Families of VHS}
\label{famVHS}

We now explain one possible motivation for looking at the stratification by the $G_{\alpha}$, related to 
Lemma \ref{lemGMBR}. 

Suppose given a smooth projective family of curves $f:X\rightarrow Y$ over a base which is allowed to be quasiprojective.
Let $X_y:= f^{-1}(y)$ be the fiber over a point $y\in Y$. Then $\pi _1(Y,y)$ acts on $M_B(X_y, r)$. The fixed points are
the representations which come from global representations on the total space $X$ of the fibration (approximately, up to considerations
involving the center of the group and so on). In the de Rham point of view there is a ``connection'' on the
relative de Rham moduli space $M_{\rm DR}(X/Y,r)\rightarrow Y$. This can be called the ``nonabelian Gauss-Manin connection''
but is also known as the system of ``isomonodromic deformation'' equations, or Painlev\'e VI for the universal family of $4$-pointed $\pp ^1$'s.
The fixed points of the above action on the Betti space correspond
to global horizontal sections of the n.a.GM connection. These have been studied by many authors.  

The global horizontal sections will often be rigid (hence VHS's) and in any case can be deformed, as horizontal sections, to VHS's.
Let $\rho (y)\in M_B(X_y,r)$ denote a global horizontal section which is globally a VHS; then we will have
$$
\rho (y)\in G^{\rm VHS}_{\alpha}(X_y)
$$
for all $y$ in a neighborhood of an initial point $y_0$. The combinatorial data corresponding to the stratum $\alpha$ will be
invariant as $y$ moves; let's assume that we can say that the index sets of our stratifications remain locally constant as a 
function of $y$, and that $\rho (y)$ stays in the ``same stratum''---although
that would require further proof. Then we have, in particular, $\rho (y)\in G_{\alpha}(y)$. The stratum $G_{\alpha}(y)$ will
not be fully invariant by the n.a.GM connection \cite{Berlinger}, so $\rho (y)$ has to lie in the subset of points which, when displaced by
the n.a.GM connection, remain in the same (i.e. corresponding) stratum. We get a system of equations constraining
the point $\rho (y)$. One can hope that in certain special cases, these equations have only isolated points as solutions
in each fiber $G_{\alpha}(y)$.
A program for finding examples of global horizontal sections $\rho$ would be to look for these isolated points 
and then verify if the family of such points is horizontal. 

\subsection{The open stratum}
\label{unitary}

At the other end of the range of possible dimensions of strata, is the unique open stratum $G_0$ in each component of 
the moduli space. When $X$ is a smooth compact curve of genus $\geq 2$, the open stratum consists of connections of the form 
$(V,\partial  + A)$ where $V$ is a polystable vector bundle, $\partial$ is the unique flat unitary connection,
and $A\in H^0(End(V)\otimes \Omega ^1_X)$. The Griffiths-transverse filtration is trivial, and the corresponding stratum
of systems of Hodge bundles is $P_0= \Uu (X,r)$, the  moduli space of semistable vector bundles on $X$. 
The Higgs stratum is just the cotangent bundle  $\widetilde{G}_0= T^{\ast}P_0$ and $G_0$ is a principal 
$T^{\ast}P_0$-torsor over $P_0$. The space $G^{\rm VHS}_0$ is just the space of unitary representations. 
The situation becomes more interesting when we look at parabolic bundles or bundles on an orbifold. 

\section{The parabolic or orbifold cases}
\label{sec-parabolic}

Let $(X,\{ x_1,\ldots , x_k\})$ be a curve with some marked points, and fix semisimple unitary conjugacy
classes $C_1,\ldots , C_k\subset GL(r)$. We can consider the various moduli spaces of representations, 
Higgs bundles, connections, or $\lambda$-connections on $U:= X-\{ x_1,\ldots , x_k\}$ with
logarithmic structure at the points $x_i$ and corresponding monodromies contained in the
conjugacy classes $C_i$ respectively. These objects correspond to  parabolic vector bundles with real parabolic
weights and $\lambda$-connection $\nabla$ respecting the parabolic 
filtration and inducing the appropriate multiple of the identity on each graded 
piece.
If in addition the conjugacy classes are assumed to be of finite order, the
parabolic weights should be rational, and our objects may then be viewed as lying on an 
orbicurve or Deligne-Mumford stack $X'$ with ramification orders $m_i$ corresponding to the common denominators of the weights
at $x_i$  \cite{Biswas} \cite{Boden} \cite{IyerSimpson}.
Denote by 
$$
M_{\rm H}(X'; C_1,\ldots , C_k) \subset M_{\rm Hod}(X'; C_1,\ldots , C_k) \supset 
M_{\rm DR}(X'; C_1,\ldots , C_k) 
\;\;\; \mbox{etc}
$$
the resulting moduli spaces. One could further assume for simplicity that there exists a global projective etale Galois covering $Z\rightarrow X'$ with Galois group $G$. The map $Z\rightarrow X$ is a ramified Galois covering
such that the numbers of branches in the  ramification points above $x_i$ are always equal to $m_i$. Local systems or other objects on $X'$ are
the same thing as  $G$-equivariant objects on $Z$. This enables an easy construction of the moduli stacks and 
spaces. 
More general constructions of moduli spaces of parabolic objects can be found in \cite{Yokogawa} \cite{Thaddeus} \cite{Nakajima} \cite{LogaresMartens}
\cite{Konno} \cite{InabaIwasakiSaito} \cite{BodenYokogawa}.

The correspondence between Higgs bundles and local systems works in this case, as can easily
be seen by pulling back to the covering $Z$ (although the analysis can also be done directly).
All of the other related structures also work the same way. For example, there is a proper Hitchin fibration 
on the space of Higgs bundles, and similarly the nonabelian Hodge filtration satisfies the positive weight property
saying that limits exist as in Lemma \ref{limitlem}. 

A first classical case appearing already in the paper of Narasimhan and Seshadri \cite{NarasimhanSeshadri}
is when there is a single parabolic point $x_1$ and the residual conjugacy class of the connection is that of the scalar matrix $d/r$.
Orbifold bundles of this type correspond to usual bundles on $X'$ of degree $d$, which can have projectively flat connections.
This led to moduli spaces in which all semistable points are stable when $(r,d)=1$. 

The semistable $\Rightarrow$ stable phenomenon occurs fairly generally in the orbifold or parabolic setting. The  $C_j$ have
to satisfy the equation $\prod {\rm det}(C_j) = 1$. However, if the eigenvalues of the individual blocks are chosen
sufficiently generally, it often happens that there is no sub-collection of $0<r'<r$ eigenvalues at each point, such that
the product over all points is equal to $1$. We call this {\em Kostov's genericity condition}.
It implies that all representations are automatically irreducible,
or in terms of vector bundles with connection or Higgs bundles, semistable objects are automatically stable.

This situation is particularly relevant for our present discussion: 

\begin{lemma}
Suppose the conjugacy classes $C_1,\ldots , C_k$ satisfy Kostov's genericity condition, then 
any vector bundle with connection $(V,\nabla )\in M_{DR}(X'; C_1,\ldots , C_k)$ has a unique partial oper 
structure, the filtration being unique up to shifting the indices. Also, in this case the 
moduli stacks $\Mm _{\cdot}(X'; C_1,\ldots , C_k)$ are smooth, and they are $\Gm$-gerbs over the corresponding moduli spaces 
$M _{\cdot}(X'; C_1,\ldots , C_k)$. 
\end{lemma}
\begin{proof}
If all semistable objects are automatically stable,
then the same is true for the associated-graded Higgs bundles. Proposition \ref{filtunique}, which works equally
well in the orbifold case, implies that the filtration is unique up to a shift. 

For the statement about moduli stacks, stable objects admit only scalar automorphisms, and the GIT construction
of the moduli spaces of stable objects yields etale-locally fine moduli, which is enough to see the gerb statement. 
For smoothness, the obstructions land in the trace-free part of  $H^2$ which
vanishes by duality. 
\end{proof}

Note that when there is a single point and the conjugacy class is scalar multiplication by a primitive $r$-th root of
unity, this is Kostov-generic, and we are exactly in the original situation considered by Narasimhan and Seshadri
corresponding to bundles of degree coprime to the rank. 

\begin{corollary}
If the conjugacy classes satisfy Kostov's genericity condition, then $M_{\rm Hod}(X'; C_1,\ldots , C_k)$ is
smooth over $\aaa ^1$, and the fixed point sets of the $\Gm$-action $P_{\alpha}$ are smooth. The projections
$G_{\alpha}\rightarrow P_{\alpha}$ and $\widetilde{G}_{\alpha}\rightarrow P_{\alpha}$ 
are smooth fibrations topologically equivalent to the 
normal bundle $N_{\widetilde{G}_{\alpha}/P_{\alpha}}\rightarrow P_{\alpha}$.
\end{corollary}
\begin{proof}
The fixed-point sets are smooth by the theory of Bialynicki-Birula. 
Let $G_{\alpha}^{\rm Hod}\subset M_{\rm Hod}(X'; C_1,\ldots , C_k)$ denote the
subset of points whose limit lies in $P_{\alpha}$. It is smooth, and the
fiber over $\lambda = 0$ is the smooth $\widetilde{G}_{\alpha}$. By the action of $\Gm$,
the map $G_{\alpha}^{\rm Hod}\rightarrow \aaa ^1$ is smooth everywhere. Let
$G_{\alpha , VHS}^{\rm Hod}\subset G_{\alpha}^{\rm Hod}$ denote the 
subspace of points which correspond to variations of Hodge structure.
The preferred-section trivialization gives
$$
G_{\alpha , VHS}^{\rm Hod}\cong P_{\alpha}\times \aaa ^1
$$
compatibly with the $\Gm$-action.

Let $T\subset G_{\alpha}^{\rm Hod}$ be a tubular neighborhood, and denote
by $T_0$ the fiber over $\lambda = 0$. We can suppose that $\partial T_0$
is transverse to the vector field defined by the action of $\rr ^{\ast}_{>0}\subset \Gm$.
Consequently, there is $\epsilon$ such that $\partial T$ is transverse to this vector
field for any $| \lambda |<2 \epsilon$.

Choose a trivialization $F:\partial T_0\times \aaa ^1\cong \partial T$ compatible with the map to $\aaa ^1$. Then define a map
$$
R : \rr ^{\ast}_{>0} \times \partial T_0 \rightarrow G_{\alpha},
$$
$$
R(t,x):= t\cdot F(x, t^{-1}).
$$
This is a diffeomorphism in the region $t> \epsilon ^{-1}/2$. 
On the other hand, let $T_{\epsilon}$  be the fiber over $\lambda = \epsilon$
and consider the map $\epsilon ^{-1}: T_ {\epsilon} \rightarrow G_{\alpha}$.
This glues in with the map $R$ defined on the region $t\geq \epsilon ^{-1}$
to give a topological trivialization of $G_{\alpha}$. Everything can be done relative to $P_{\alpha}$
so we get a homeomorphism between  $G_{\alpha}\rightarrow P_{\alpha}$ and the normal bundle of $P_{\alpha}$ in $G_{\alpha}$.
The same happens for $\widetilde{G}_{\alpha}\rightarrow P_{\alpha}$.
\end{proof}

\subsection{The Hodge type of the open stratum}
\label{hodgetype}
The open stratum consists of variations of Hodge structure of a certain type.
This type can change as we move the conjugacy classes. For example, the set of vectors of conjugacy classes $(C_1,\ldots , C_k)$
for $k$ points in $\pp ^1$, for which there exists a unitary representation, is defined by some inequalities
on the logarithms of the eigenvalues \cite{AW} \cite{Belkale} \cite{BiswasUnitary} \cite{BodenYokogawa} \cite{TelemanWoodward}. 

If we fix a collection of partitions at each point, then the space of unitary conjugacy classes
corresponding to those partitions is a real simplex, and the set of finite-order conjugacy classes is the set of
rational points therein. The set of non-Kostov-generic points is a union of hyperplanes, so the set of Kostov-generic
points decomposes into a union of open chambers bounded by pieces of hyperplanes (in particular, the chambers are
polytopes).

Suppose that we know that the varieties $M_{\cdot}(X';C_1,\ldots , C_k)$ are connected---this is the
case, for example, if at least one of the conjugacy classes has distinct eigenvalues \cite{KostovConnected}.
Then there is a single open stratum. If the genericity condition is satisfied so
all objects are stable, then the partial oper structure at a general point is unique, 
and its Hodge type (up to translation) depends only on the conjugacy classes. 
Furthermore, given a Kostov-generic collection of unitary conjugacy classes 
$(C_1,\ldots , C_k)$  and a parabolic system of Hodge bundles
with these conjugacy classes, if we perturb slightly the parabolic weights then the object remains stable.
This shows that the Hodge type for the unique open stratum is constant on the Kostov-generic chambers. 

A limit argument should show that if we fix a particular Hodge type $\{ h^{p,q}\}$, then the set of 
vectors of conjugacy classes $(C_1,\ldots , C_k)$ for which the points in the open stratum admit
a partial oper structure with given $h^{p,q}$, is closed. We don't do that proof here, though---let's just assume
it's true. It certainly holds for the unitary case $h^{0,0}=r$. This closed set is then a closed polytope.
For the unitary case, it is the polytope found and studied by 
Boden, Hu, Yokogawa \cite{BodenHu} \cite{BodenYokogawa}, Biswas \cite{BiswasUnitary} \cite{BiswasUnitary2}, 
Agnihotri, Woodward \cite{AW}, Belkale \cite{Belkale}.  The boundary is a union of pieces of the
non-Kostov-generic hyperplanes. 

The fact that the Hodge type varies in a way which is locally constant over polytopes fits into the
general philosophy which comes out of the results of Libgober \cite{Libgober} and Budur \cite{Budur}.
There is a direct relationship in the rigid case: by Katz's algorithm the unique point in a rigid moduli
space can be expressed motivically, and the Hodge type is locally constant on polytopes by the theory of
\cite{Libgober} \cite{Budur}.

There seems to be a further interesting phenomenon going on at the boundary of these polytopes: the complex variations of Hodge
structure corresponding to the fixed point limits of points in the open stratum, become automatically reducible. 
In the unitary case this is pointed out in several places in the references
(see for example Remark (4) after Theorem 7 on page 75 of \cite{Belkale}, or also
the last phrase of Theorem 3.23 of \cite{BiswasUnitary2}): 
irreducible unitary representations exist only in the interior of the polytope. A sketch of proof in general is
to say that if there is a stable system of Hodge bundles with given parabolic structure then the parabolic structure
can be perturbed keeping stability and keeping the same Hodge type (it is the same argument as was used above at the interior points
of the chambers, which is indeed essentially the same as in the references). 

In other words, at the walls between the chambers, points in the open stratum become gr-semistable but not gr-stable. 
Now, the non-uniqueness of the partial oper structure, Proposition \ref{filtunique}, is exactly what allows the Hodge type to jump. 

For generic $(C_1,\ldots , C_k)$ where the moduli space is smooth, the space of complex variations of Hodge structure
corresponding to the open stratum is a real form of the Betti moduli space.
It is defined algebraically in $M_{B}(X;C_1,\ldots , C_k)$ as a connected component of the space of representations into some
$U(p,q)$. 

However, at boundary points $(C_1,\ldots , C_k)$ between the different Hodge polytopes this subspace is concentrated on the singular
locus of reducible representations (whereas in most cases there still exist irreducible representations). The open stratum is
presented as a cone over a fixed point set inside the singular locus. The reader will be convinced that this kind of thing can happen
by looking at the case of the real circle $x^2+y^2=t$: for $t>0$ it is a real form of the complex quadratic curve, but at $t=0$
the real points are just the singularity of two crossing complex lines. It is an interesting question to understand what the degeneration
of the real form of the smooth space looks like for the case of $M_{B}(X;C_1,\ldots , C_k)$.

\section{Deformation theory}
\label{deformations}

The deformation theory follows \cite{BiswasRamanan} \cite{Bottacin} \cite{Markman} and others. The discussion extends without further mention
to the case where $X$ is an orbifold. Further work is needed \cite{YokogawaDef} \cite{Bottacin}
for more general parabolic cases corresponding for example to non-unitary conjugacy classes of
local monodromy operators.  

Suppose $(V,\nabla )$ is an irreducible connection. 
Consider the complex 
$$
End(V)\otimes \Omega ^{\cdot}_X := [End(V)\stackrel{d}{\rightarrow}End(V)\otimes \Omega ^{1}_X].
$$
The space of infinitesimal deformations, or the tangent space to the moduli stack $\Mm _{\rm DR}$ at $(V,\nabla )$,
is
$$
Def(V,\nabla )= \hhh ^1(End(V)\otimes \Omega ^{\cdot}_X ).
$$
Given a Griffiths-transverse filtration $F^{\cdot}$ we get a decreasing filtration of the complex
$End(V)\otimes \Omega ^{\cdot}_X$ 
defined by
$$
F^p(End(V)) := \{ \varphi \in End(V), \;\;\; \varphi : F^qV \rightarrow F^{p+q}V\} , 
$$
$$
F^p(End(V)\otimes \Omega ^{1}_X):= F^{p-1}(End(V))\otimes \Omega ^{1}_X.
$$
The $F^p(End(V)\otimes \Omega ^{\cdot}_X)$ are subcomplexes, so we can take the spectral sequence of the filtered complex 
$(End(V)\otimes \Omega ^{\cdot}_X,F^{\cdot})$. It induces a filtration $F^{\cdot}Def(V,\nabla )$ and the spectral sequence is 
\begin{equation}
\label{specseq}
\hhh ^i(Gr^{p}_F(End(V)\otimes \Omega ^{\cdot}_X)) = E^{p,i-p}_1 \Rightarrow Gr^{p}_F\hhh ^i(End(V)\otimes \Omega ^{\cdot}_X ).
\end{equation}
The obstruction theory is controlled by the $\hhh ^2$ of the trace-free part, so if $(V,\nabla )$ is irreducible then 
the deformations are unobstructed, there are no problems with extra automorphisms, and the moduli space is smooth. 

The limiting system of Hodge bundles $(Gr_F(V),\theta )$ has its own deformation theory in the world of Higgs bundles.
The complex $End(Gr_F(V))\otimes \Omega ^{\cdot}_X$ is made as above, but the differentials using $\theta$ are 
$\Oo_X$-linear maps. The tangent space to the moduli stack $\Mm _{\rm H}$ at $(Gr_F(V),\theta )$ is calculated as
$$
Def(Gr_F(V),\theta ) = \hhh ^1(End(Gr_F(V))\otimes \Omega ^{\cdot}_X ).
$$
The complex $End(Gr_F(V))\otimes \Omega ^{\cdot}_X$ has a direct sum decomposition, and indeed it is the associated-graded of the previous one:
$$
End(Gr_F(V))\otimes \Omega ^{\cdot}_X = Gr_F(End(V)\otimes \Omega ^{\cdot}_X).
$$
This tells us that the $E^{\cdot , \cdot}_1$ term of the previous spectral sequence corresponds to the deformation theory of $(Gr_F(V),\theta )$.

\begin{lemma}
\label{degen}
If $(Gr_F(V),\theta )$ is stable then 
the spectral sequence \eqref{specseq} degenerates at $E_1$ with
\begin{equation}
\label{grdecomp}
Gr^{p }_F (Def(V,\nabla )) = 
\hhh ^1 \left( Gr^p_F(End(V)\stackrel{\theta}{\rightarrow}Gr^{p-1}_F(End(V)\otimes \Omega ^{\cdot}_X \right) .
\end{equation}
Let $T(G_{\alpha}/P_{\alpha})$ denote the relative tangent bundle of the fibration $G_{\alpha} \rightarrow P_{\alpha}$
where it is smooth. Then 
the two middle subspaces of the filtration at $p=0,1$ are interpreted
in
a geometric way as
$$
F^0 Def(V,\nabla ) = T(G_{\alpha})_{(V,\nabla )},\;\;\; \mbox{and}\;\;\; F^1 Def(V,\nabla ) = T(G_{\alpha}/P_{\alpha})_{(V,\nabla )}.
$$
\end{lemma}
\begin{proof}
Since $(Gr_F(V),\theta )$ is stable,
$$
\hhh ^0(Gr_F(End(V)\otimes \Omega ^{\cdot}_X))=\cc , \;\;\;
\hhh ^2(Gr_F(End(V)\otimes \Omega ^{\cdot}_X)) = \cc
$$
and these are the same as at the limit of the spectral sequence. By invariance of the Euler characteristic
the same must be true for $\hhh ^1$ so the spectral sequence degenerates. The geometric interpretation may be
seen by looking at a cocycle description. 
\end{proof}

\subsection{At a variation of Hodge structure}
\label{hodgedef}

Suppose $(V,\nabla ) \in G^{\rm VHS}_{\alpha}$ is an irreducible complex variation of Hodge structure.
Then $End(V)$ is a
real VHS of weight $0$, independent of the scalar choice of polarization on $V$. The hypercohomology $Def(V,\nabla )= \hhh ^1(End(V)\otimes \Omega ^{\cdot}_X)$
is a real Hodge structure of weight $1$, with the same Hodge filtration as defined above and which we write as
$$
H_{\rr}\subset Def(V,\nabla )\cong \bigoplus \hhh ^{k,1-k}.
$$
This decomposition splits the Hodge filtration and is naturally isomorphic to the decomposition \eqref{grdecomp}.  
The
real subspace $H_{\rr}$ is the space of deformations of the monodromy representation preserving the polarization form, i.e. as a representation in
the real group $U(p,q)$. In other words it is the tangent space to $M_B(X,r)_{\rr}$ in the notations of Lemma \ref{lemGMBR}.

\noindent
{\em Proof of Lemma \ref{lemGMBR}:}
Both $G^{\rm VHS}_{\alpha}$ and $G_{\alpha} \cap M_B(X,r)_{\rr}$  have the same tangent spaces at $(V,\nabla )$.  
The tangent space of $G_{\alpha}$ is the sum $F^0=\bigoplus_{k\geq 0} \hhh ^{k,1-k}$. The intersection of $F^0$ with
$H_{\rr}$ is $(\hhh ^{1,0}\oplus \hhh ^{0,1})_{\rr}$ which is the tangent space of $G^{\rm VHS}_{\alpha}$.
\eop

The symmetry of Hodge numbers extends to other points too: 

\begin{corollary}
\label{symmetry}
For any gr-stable point $(V,\nabla )$, not necessarily a VHS, let $F^{\cdot}$ be the unique (up to shift) partial oper structure. Then
the induced filtration on the tangent space of the moduli stack at $(V,\nabla )$ satisfies
$$
\dim Gr_F^p Def(V,\nabla ) = \dim Gr_F^{1-p} Def(V,\nabla ).
$$
\end{corollary}
\begin{proof} 
Let $(V',\nabla ')$ denote the VHS corresponding to the system of Hodge bundles $(Gr_F(V),\theta )$. 
Note that $(Gr_F(V') , \theta ')\cong (Gr_F(V),\theta )$,
so
$$
Gr_F^p Def(V,\nabla ) \cong \hhh ^1(End(Gr_F(V))\otimes \Omega ^{\cdot}_X) \cong Gr_F^pDef(V',\nabla ').
$$
Therefore we can write
$Gr_F^p Def(V,\nabla ) \cong \hhh ^{p,1-p} (End(V')\otimes \Omega ^{\cdot}_X)$, and these spaces
are associated to a real Hodge structure of weight one, so we get the claimed symmetry of dimensions. 
\end{proof}

In the Hitchin moduli space the tangent space at a point $(E,\theta )\in P_{\alpha}$ decomposes as a direct sum under
the $\Gm$-action, and this decomposition is compatible with the filtration, so it gives a splitting:
$$
T(\Mm _{\rm H})_{(E,\theta )} \cong \bigoplus Gr_F^p T(\Mm _{\rm H})_{(E,\theta )}\cong \bigoplus \hhh ^{p, 1-p}.
$$
The tangent space to the fixed point set is
$Gr_F^0 T(\Mm _{\rm H})_{(E,\theta )}= T(P_{\alpha})_{(E,\theta )}$,
whereas the pieces $p<0$ may be identified as the ``outgoing'' directions for the $\Gm$-action.

If $P_{\alpha}= P_0$ is the lowest piece corresponding to the open stratum $G_0\subset M$ then there are no outgoing directions,
so the terms of Hodge type $(p,1-p)$ vanish for $p<0$. By symmetry they vanish for $p>1$, which says that
$Gr_F^p T(\Mm _{\rm H})_{(E,\theta )} = 0$ unless $p=0,1$; the same holds for $\Mm _{\rm DR}$.
Furthermore, the two pieces are dual by the symplectic form (see below), so
we can identify
$$
\hhh ^{0,1} = Gr ^0_F T(\Mm _{\rm H})_{(E,\theta )}\cong T(P_0)_{(E,\theta )} , 
$$
$$
\hhh ^{1,0} = Gr ^1_F T(\Mm _{\rm H})_{(E,\theta )}\cong T^{\ast}(P_0)_{(E,\theta )} . 
$$
We even have an isomorphism of fibrations
$$
\begin{array}{ccc}
\widetilde{G}_0 & & T^{\ast}P_0 \\
\downarrow & \cong & \downarrow \\
P_0 & & P_0 ,
\end{array}
$$
whereas $G_0\rightarrow P_0$ is a twisted form of the same fibration (it is a torsor). These observations are elementary and classical when $X$ is
a smooth projective curve, however they also hold in the parabolic or orbifold case when the weights are generic so that all points are $gr$-stable. 

This includes cases where there are no stable parabolic vector bundles, but the lowest  stratum $P_0$ corresponds to variations of Hodge structure 
of nontrivial Hodge type.

\subsection{Lagrangian property for the fibers of the projections $G_{\alpha}\rightarrow P_{\alpha}$}
\label{sec-lagrange}

There is a  natural symplectic form  on the tangent bundle to the moduli space over the open subset of stable points,
given in cohomological terms as the cup product
$\cup :\hhh ^1 \otimes \hhh ^1\rightarrow \hhh ^2 = \cc$. 

In the moduli space of connections on a smooth projective curve, the open stratum $G_0$
fibers over the moduli space $P_0$
of semistable vector bundles. In the Hitchin moduli space, under the isomorphism  $\widetilde{G}_0\cong T^{\ast}P_0$, 
the symplectic form is equal to the standard one \cite{BiswasRamanan},
in particular the fibers are lagrangian.  

It has been noticed by several authors \cite{Hain} \cite{Bottacin} \cite{Szabo} \cite{Lekaus} \cite{Aidan} that
the fibers of the projection $G_0\rightarrow P_0$ (over stable points)
are similarly lagrangian subspaces of the moduli space of connections.   

We point out that the lagrangian property extends to the fibers in all strata,
at gr-stable points. 

\begin{lemma}
\label{lagrange}
Suppose $p=(E,\theta )\in P_{\alpha}$ is a stable system of Hodge bundles. Then the fiber
$L_p$ of the projection $G_{\alpha}\rightarrow P_{\alpha}$ over $p$ is a lagrangian subspace of $M_{\rm DR}$.
\end{lemma}
\begin{proof}
Fix a point $(V,\nabla )$ which is gr-stable, hence also stable itself. 
By Lemma \ref{degen}, 
the tangent space to the fiber of $G_{\alpha}/ P_{\alpha}$ is the subspace $F^1 \hhh ^1 \subset Def(V,\nabla )$. 

In degree two, $F^1\hhh ^2 (End(V)\otimes \Omega ^{\cdot}_X )$ is the full hypercohomology space $\hhh ^2 = \cc $
whereas $F^2\hhh ^2 = 0$. On the other hand, the Hodge filtration is compatible with cup product,
so 
$$
\cup : F^1\hhh ^1 \rightarrow F^1\hhh ^1 \rightarrow F^2\hhh ^2 = 0,
$$
which is to say that $F^1\hhh ^1 \subset Def (V,\nabla )$ is an isotropic subspace for the symplectic form.
The symmetry property Corollary \ref{symmetry} readily implies that
$F^1\hhh ^1$ has half the dimension, so it is lagrangian. 
\end{proof}

A natural problem is to understand the relationship between the lagrangian fibers on different strata.

\begin{question}
\label{foliation}
Do the lagrangian fibers of the projections fit together into a smooth lagrangian foliation with closed leaves?
\end{question}

This is of course true within a given stratum $G_{\alpha}$; does it remain true as $\alpha$ varies?

A heuristic argument for why this might be the case is that if a fiber of $G_{\alpha}\rightarrow P_{\alpha}$
were  not closed in $M_{\rm DR}$, this might lead to a projective curve in $M_{\rm DR}$ which cannot exist since
$M_{\rm B}$ is affine.   

A similar closedness or properness property was proven in \cite{Szabo} for the case of Fuchsian equations. 

On the other hand,
the analogous statement for $M_{\rm H}$ is not true: the fibers of $\widetilde{G}_{\alpha}\rightarrow P_{\alpha}$ definitely do sometimes have nontrivial
closures in $M_{\rm H}$, indeed any fiber contained in the compact nilpotent cone will be non-closed. 

If the answer to Question \ref{foliation} is affirmative, it could be useful for the context of geometric Langlands, where the full moduli stack
$Bun_{GL(r)}$ might profitably be replaced by the algebraic space of leaves of the foliation. They share the open set of semistable bundles. 
This could help in the ramified case \cite{GukovWitten}---when there are no semistable bundles, it would seem logical to consider the lowest stratum $P_0$
parametrizing variations of Hodge structure.

\subsection{Nestedness of the stratifications}
\label{nestedness}

We are given a stratification 
or disjoint decomposition into locally closed subsets
$M = \coprod _{\alpha} G_{\alpha}$.
Say that it is {\em nested} if there is a partial order on the index set such that
$$
\overline{G}_{\alpha} - G_{\alpha} = \coprod _{\beta < \alpha} G_{\beta}.
$$
If this is the case, the partial order is defined by the condition 
$$
\beta \leq \alpha \Leftrightarrow G_{\beta}\subset \overline{G}_{\alpha}.
$$
The {\em arrangement} of the strata is the partially ordered index set.

\begin{conjecture}
\label{nested}
The stratifications of $M_{\rm DR}(X,r)$ and $M_{\rm H}(X,r)$ defined in Proposition \ref{strata} are nested,
and the arrangements of their strata are the same via the identification of each stratum with
the corresponding fixed point set in $M_{\rm H}(X,r)$.
\end{conjecture}

The arrangement of the strata in the Hitchin moduli space has been studied in \cite{BiswasGomez}, and for rank $3$ the
connected components of the fixed point sets have been classified in \cite{GPGM}. 

\begin{theorem}
\label{nested2}
Conjecture \ref{nested} is true for bundles of rank $r=2$ on a smooth projective curve
$X$ of genus $g\geq 2$.
\end{theorem}
\begin{proof}
Hitchin \cite{Hitchin} identifies explicitly 
the connected components $P_e$ of the fixed point set $M_{\rm H}(X,2)^{\Gm}$.
These are indexed by an integer $0\leq e \leq g-1$, where $P_0$ is the space of semistable
vector bundles and  for $e>1$, $P_e$ is the space of systems of Hodge bundles which are direct sums of 
line bundles of degrees $e$ and $-e$.  A deformation theory argument will show that starting with
a point in $P_e$ we can deform it to a family of $\lambda$-connections which project
into the next stratum $P_{e-1}$.

At $e=0$, the space $P_0$ is
the moduli space of rank two semistable vector bundles on $X$, which is sometimes denoted
$\Uu _X(2)$, and is known to be an irreducible variety. 

For $e>0$ the space $P_e$ parametrizes Higgs bundles of the form
$$
E=E^0\oplus E^1 , \;\;\;\; \theta : E^1\rightarrow E^0\otimes \Omega ^1_X,
$$
where $E^0$ and $E^1$ are line bundles of degrees $-e$ and $e$ 
respectively. 
We require $e>0$ because a Higgs bundle of the above form with $deg(E^i)=e=0$ would be semistable but not stable,
in the same $S$-equivalence class as the polystable vector bundle
$E^0\oplus E^1$ which is a point in $P_0$.  

The map $\theta$ is a section of the line bundle $(E^1)^{\ast}\otimes E^0\otimes \Omega ^1_X$
of degree $2g-2-2e$, so $e\leq g-1$ and in the case of equality, $\theta$ is an isomorphism.
Let $D$ be the divisor of $\theta$; it is an effective divisor of degree $2g-2-2e$, and the  space
of such is $Sym ^{2g-2-2e}(X)$. Given $D$ and $E^1 \in Pic ^e(X)$ then  $E$ is
determined by $E^0= E^1\otimes TX \otimes \Oo _X(D)$. Thus,
$P_e \cong Sym ^{2g-2-2e}(X)\times Pic^e(X)$
which is a smooth irreducible variety. All points are stable. 

The strata of the oper stratification are enumerated as $G_e\subset M_{\rm DR}(X,2)$ for $0\leq e \leq g-1$; the
corresponding strata of $M_{\rm H}(X,2)$ are denoted $\widetilde{G}_e$. The upper (smallest) stratum $e=g-1$ is the case
of classical opers. 

To show the nested  property, it suffices to show that $G_e \subset \overline{G_{e-1}}$ for $e\geq 2$.
Note that for $e=1$ this is automatic since $G_0$ is an open dense subset so $\overline{G_0} = M_{\rm DR}$.
We will consider a point $(V,\nabla )\in G_e$ and deform it to a family $(V_t,\nabla _t)\in G_{e-1}$ 
for $t\neq 0$, with limit $(V,\nabla )$ as $t\rightarrow 0$. 

Having $(V,\nabla , F^{\cdot})\in G_e$ means that there is an exact sequence
\begin{equation}
\label{Ees}
0\rightarrow E^1\rightarrow V \rightarrow E^0 \rightarrow 0
\end{equation}
with $E^1= F^1$ and $E^0 = F^0/F^1$. 
The limit point in $P_e$  is $E=E^1\oplus E^0$ with $\theta := \nabla : E^1\rightarrow E^0\otimes \Omega ^1_X$.
Choose a point $p\in X$ and let $L:= E^1(-p)$. Let $\varphi : L\rightarrow V$ be the inclusion. Consider the functor
of deformations of $(V,\nabla , L, \varphi )$. It is controlled by the hypercohomology of the complex 
$$
C^{\cdot}:= End(V)\oplus Hom(L,L) \rightarrow End(V)\otimes \Omega ^1_X \oplus Hom (L,V)
$$
where the differential is the matrix
$$
\left(
\begin{array}{cc}
\nabla & 0 \\
-\circ \varphi & \varphi \circ -
\end{array}
\right) . 
$$
There is a long exact sequence
$$
\ldots \rightarrow \hhh ^i(Hom(L,L)\rightarrow Hom(L,V)) \rightarrow \hhh ^i(C^{\cdot}) \rightarrow 
\hhh ^i(End(V)\otimes \Omega ^{\cdot}_X)\rightarrow \ldots .
$$

The fact that our filtration is $gr$-stable implies that the spectral sequence
associated to the filtered complex $(End(V)\otimes \Omega ^{\cdot}_X, F^{\cdot})$ degenerates at 
$E^{\cdot , \cdot} _1$. The VHS on $End(Gr_F(V))$ has Hodge types $(1,-1) + (0,0) + (-1,1)$
so its $\hhh ^1$ has types $(2,-1)+ \ldots + (-1,2)$. The associated-graded of  $\hhh ^1(End(V)\otimes \Omega ^{\cdot}_X)$ has pieces
$$
Gr^2_F\hhh ^1 = H^0(Hom (E^0, E^1)\otimes \Omega ^1_X), 
$$
$$
Gr^1_F\hhh ^1 =\hhh ^1 (Hom (E^0,E^1) \rightarrow ( Hom (E^0,E^0)\oplus Hom (E^1, E^1))\otimes \Omega ^1_X), 
$$
$$
Gr^0_F\hhh ^1 =\hhh ^1 (Hom (E^1,E^1)\oplus Hom (E^0, E^0) \rightarrow Hom (E^1,E^0)\otimes \Omega ^1_X), 
$$
$$
Gr^{-1}_F\hhh ^1 = H^1(Hom (E^1, E^0)).
$$
The degeneration gives a surjection $\hhh ^1 \rightarrow \!\!\!\!\! \rightarrow Gr^{-1}_F\hhh ^1$, 
in other words
$$
\hhh ^i(End(V)\otimes \Omega ^{\cdot}_X) \rightarrow H^1(Hom (E^1, E^0))\rightarrow 0.
$$
The definition $L= E^1(-p)$ gives an exact sequence of the form 
\begin{equation}
\label{ptseq}
0\rightarrow Hom (E^1,E^0) \rightarrow Hom (L, E^0) \rightarrow \cc _p \rightarrow 0.
\end{equation}
Note that $deg (L)>0$ but $deg(E^0)<0$ so $H^0(Hom (L, E^0))=0$, 
hence the long exact sequence for \eqref{ptseq} becomes
$$
0\rightarrow \cc \rightarrow H^1(Hom (E^1, E^0)) \rightarrow H^1(Hom (L, E^0)) \rightarrow 0 .
$$
Putting these two together we conclude that the map 
$$
\hhh ^1(End(V)\otimes \Omega ^{\cdot}_X) \rightarrow H^1(Hom (L, E^0))
$$ 
is surjective, but there is an element in its kernel which maps to something nonzero in $H^1(Hom (E^1, E^0))$. 
Another exact sequence like \eqref{ptseq} gives a surjection 
$$
H^1(\Oo _X)= H^1(Hom (L,L)) \rightarrow H^1(Hom (L, E^1)) \rightarrow 0.
$$
The long exact sequence for $Hom$ from $L$ to the short exact sequence \eqref{Ees} gives
$$
H^1(Hom (L,E^1))\rightarrow H^1(Hom (L,V))\rightarrow H^1(Hom (L,E^0))\rightarrow 0 .
$$
Therefore the map 
\begin{equation}
\label{themap}
H^1(Hom (L,L))\oplus \hhh ^1(End(V)\otimes \Omega ^{\cdot}_X) \rightarrow H^1(Hom (L,V))
\end{equation}
is surjective but not an isomorphism. This gives vanishing of 
the obstruction theory for deforming $(V,\nabla , L, \varphi )$. 
And the map \eqref{themap} has an element in its kernel which maps to something
nonzero in $H^1(Hom (E^1, E^0))$. Hence there is a deformation of $(V,\nabla , L, \varphi )$ which doesn't extend to
a deformation of $E^1\subset V$. A one-parameter family of $(V_t, \nabla _t, L_t, \varphi _t)$ thus gives the deformation 
from $(V,\nabla )$ into the next lowest stratum $G_{e-1}$ where $L_t$ of degree $e-1$ will be the destabilizing subsheaf.
\end{proof}

\section{Principal objects}
\label{principal}

Ramanan would want us to consider also principal bundles for arbitrary reductive structure group $G$.
Notice that the category of partial oper structures, while posessing tensor product and dual operations,
is not tannakian because morphisms of filtered objects need not be strict. The passage to principal
objects should be done by hand.
One can construct the various moduli spaces 
of principal bundles with connection, principal Higgs bundles, and the
nonabelian Hodge moduli space of principal bundles with $\lambda$-connection 
$$
\begin{array}{ccccc}
M_{\rm H}(X,G) & \subset & M_{\rm Hod}(X,G) & \supset & M_{\rm DR}(X,G) \\
\downarrow & & \downarrow & & \downarrow \\
\{ 0\} &\subset &  \aaa ^1 & \supset & \{ 1 \} 
\end{array}.
$$
In this section we indicate how to prove that
the limit $\lim _{t\rightarrow 0}t\cdot p$ exists for any $p\in M_{\rm Hod}(X,G)$. The discussion using
harmonic bundles will be more technical than the previous sections of the paper. 

Embedding $G\hookrightarrow GL(r)$,
it suffices to show that the map of moduli spaces $M_{\rm Hod}(X,G)\rightarrow M_{\rm Hod}(X,GL(r))$ is finite
and then apply Lemma \ref{limitlem}. Such a finiteness statement was considered for $M_{\rm H}$ and $M_{\rm DR}$
in \cite{moduli} along with the construction of the homeomorphism between these two spaces.
The same statements were mentionned for $M_{\rm Hod}$ in \cite{hfnac}, however the discussion there was
inadequate. The main issue is to prove that the map from the moduli space of framed harmonic bundles to
the GIT moduli space $M_{\rm Hod}(X,G)$ is proper. The distinct arguments at $\lambda = 0$ and $\lambda = 1$
given in \cite{moduli} don't immediately generalize to intermediate values of $\lambda$. 
This is somewhat similar to---and inspired by---the convergence questions treated recently by Mochizuki in \cite{MochizukiLambda}. 
We give an argument based on the topology
of the moduli space together with its Hitchin map; it would be interesting to have a more direct argument with explicit estimates. 

Fix a basepoint $x\in X$, and let
$R_{\rm Hod}(X,x,G)$ be the parameter variety for $(\lambda , P,\nabla , \zeta )$ where $\lambda \in \aaa ^1$, 
$P$ is a principal $G$-bundle with $\lambda$-connection $\nabla$ (such that $(P,\nabla )$ is semistable with vanishing Chern classes)
and $\zeta : G\cong P_x$ is a frame for $P$ at the point $x$. The group $G$ acts and 
$$
R_{\rm Hod}(X,x,G)\rightarrow M_{\rm Hod}(X,G)
$$
is a universal categorical quotient. 
Indeed, $R_{\rm Hod}(X,x,G)$ is constructed first as a closed subscheme of $R_{\rm Hod}(X,x,GL(r))$ for a 
closed embedding $G\subset GL(r)$ by the same construction as in \cite[II, \S 9]{moduli}, which treated the Higgs case.
This was based on  \cite[I,\S 4]{moduli} which treats the general $\Lambda$-module case so the discussion transposes to
$M_{\rm Hod}$.
One can choose a $GL(r)$-linearized line bundle on $R_{\rm Hod}(X,x,GL(r))$ for which every point is semistable,
which by restriction gives a $G$-linearized line bundle on $R_{\rm Hod}(X,x,G)$ for which every point is semistable.
Mumford's theory then gives the universal categorical quotient.

Fix a compact real form $J\subset G$. For any $\lambda \in \aaa ^1$ a {\em harmonic metric} on 
a principal bundle with $\lambda$-connection $(P,\nabla )$ is a $\Cc ^{\infty}$ reduction of structure group $h\subset P$
from $G$ to $J$ satisfying some equations (see \cite{MochizukiLambda} for example) 
which interpolate between the Yang-Mills-Higgs equations at $\lambda = 0$ and the 
harmonic map equations at $\lambda = 1$. Let $R^J_{\rm Hod}(X,x,G) \subset R_{\rm Hod}(X,x,G)$ denote the subset of 
$(\lambda , P,\nabla , \zeta )$ such that there exists a harmonic metric $h$ compatible with the frame $\zeta$ at $x$,
in other words $\zeta : J\rightarrow h_x$. This condition
fixes $h$ uniquely. 

\begin{lemma}
The map 
$R^J_{\rm Hod}(X,x,G)\rightarrow M_{\rm Hod}(X,G)$
is proper, and induces homeomorphisms 
$$
M_{\rm H}(X,G)\times \aaa ^1\cong R^J_{\rm Hod}(X,x,G)/J\cong  M_{\rm Hod}(X,G) .
$$
\end{lemma}
\begin{proof}
The moduli space of harmonic $\lambda$-connections  has a natural product structure 
$R^J_{\rm Hod}(X,x,G) \cong Har^J(X,G)\times \aaa ^1$ where $Har^J(X,G)$ is the space of framed harmonic $G$-bundles;
and the topological quotient by the action of $J$ on the framing is again a product $(Har^J(X,G)/J)\times \aaa ^1$. 
The second homeomorphism follows from the properness statement, by the discussion in \cite{moduli}. There 
the properness was proven at $\lambda = 0,1$. At $\lambda = 0$ we get $Har^J(X,G)/J\cong M_{\rm H}(X,G)$, which gives the
first homeomorphism. At $\lambda = 1$ we get $Har^J(X,G)/J\cong M_{\rm DR}(X,G)$. 

Fix $G\subset GL(r)$ such that $J\subset U(r)$, then  
$R^J_{\rm Hod}(X,x,G)$ is a closed subset of $R^{U(r)}_{\rm Hod}(X,x,GL(r))$. 
Using this, one can show that the lemma for $GL(r)$ implies the lemma for $G$. 

Suppose now $G=GL(r)$, with compact subgroup $J=U(r)$. The map on the subset of stable points
$R^J_{\rm Hod}(X,x,r)^s \rightarrow M_{\rm Hod}(X,r)^s$
is proper, using the fact that $M_{\rm Hod}(X,r)^s$ is a fine  moduli space,
plus the main estimate for the construction of Hermitian-Einstein harmonic metric solutions
by the method of Donaldson's functional. This estimate is explained for the case of $\lambda$-connections
in \cite{MochizukiLambda}. 

Given a polystable point $(\lambda , V, \nabla )\in M_{\rm Hod}(X,r)$ we can associate a polystable Higgs bundle $(E,\theta )$
in the preferred section corresponding to the harmonic bundle associated to any harmonic metric on $(V, \nabla )$.
The $(E,\theta )$ is unique up to isomorphism, in particular the value of the Hitchin map $\Psi = \det (\theta - t) \in \cc ^N$ is well-defined.
This gives a set-theoretically defined map $\Psi : M_{\rm Hod}(X,r)\rightarrow \cc ^N$. For a sequence of points $\rho _i\in R^J_{\rm Hod}(X,x,r)$,
there is a convergent subsequence if and only if the sequence $\Psi [\rho _i]$ contains a bounded subsequence. Hence, in order to prove
properness of the map $R^J_{\rm Hod}(X,x,r) \rightarrow M_{\rm Hod}(X,r)$ it suffices to prove that the function $\Psi$ is locally bounded on
$M_{\rm Hod}(X,r)$. This is obvious on the fiber $\lambda = 0$ where the Hitchin map $\Psi$ is an algebraic map. On the fiber
$\lambda = 1$, an argument using the characterization of harmonic metrics as ones which minimize the energy $\| \theta \| ^2 _{L^2}$
again shows that $\Psi$ is locally bounded. In particular, the lemma holds over $\lambda = 0$ and $\lambda = 1$, indeed this was the proof
of \cite{moduli} for the homeomorphism $M_{\rm H}(X,r)\cong M_{\rm DR}(X,r)$.

Properness over the set of stable points means that the map $M_{\rm Hod}(X,r)^s\rightarrow M_{\rm H}(X,r)^s$ is continuous,
hence $\Psi$ is continuous over $M_{\rm Hod}(X,r)^s$. Similarly, it is continuous on any stratum obtained by fixing the
type of the decomposition of a polystable object into isotypical components. A corollary is that for any point
$[(\lambda , V,\nabla )]\in M_{\rm Hod}(X,r)$, the function  
$\cc ^{\ast} \ni t \mapsto \Psi (t\lambda , V, t\nabla )$
is continuous.

Suppose we have a sequence of points $p_i\rightarrow q$ converging in $M_{\rm Hod}(X,r)$, but where $|\Psi (p_i)|\rightarrow \infty$.
Assume that they are all in the same fiber $M_{\lambda}$ over a fixed value $0\neq \lambda \in \aaa ^1$. 
The points $\lambda ^{-1}p_i$ converge to $\lambda ^{-1}q$ in $M_{\rm DR}(X,r)$, so (by the energy argument referred to above) 
we have a bound $|\Psi (\lambda  ^{-1}p_i)|\leq C_1$. 
Fix a curve segment $\gamma \subset \cc$ joining $\lambda ^{-1}$ to $1$ but  not passing through $0$. 
The function $t\mapsto \Psi (tq)$ is continuous by the previous paragraph, so there is a bound $|\Psi (tq)|\leq C_2$ for $t\in \gamma$. 
On the other hand, again by the continuity of the previous paragraph, for any  constant $C > C_1$ there exists a sequence of points $t_i\in \gamma$ 
such that $| \Psi (t_ip_i) | = C$. Possibly going to a subsequence, we  can assume that $t_ip_i\rightarrow q'$ as a limit of harmonic bundles.
Continuity of the Hitchin map on $M_{\rm H}$ says that $| \Psi (q') | = C$. The map from the space of harmonic bundles to $M_{\rm Hod}$ is
continuous so the limit $t_ip_i\rightarrow q'$ also holds in $M_{\rm Hod}$. On the other hand, we can assume $t_i\rightarrow t$ in $\gamma$
(again possibly after going to a subsequence), which gives $t_ip_i\rightarrow tq$ in $M_{\rm Hod}$. Separatedness of the scheme 
$M_{\rm Hod}$ implies that the topological space is Hausdorff, so $tq=q'$. If $C> C_2$ this contradicts the bound $|\Psi (tq)|\leq C_2$ for $t\in \gamma$.
We obtain a contradiction to the assumption $|\Psi (p_i)|\rightarrow \infty$, so we have proven that $|\Psi (p_i)|$ is locally bounded. 

In the fiber over each fixed $\lambda \in \aaa ^1$,
this shows properness, hence the homeomorphism statements, hence that $\Psi$ is continuous. 
Now using the connectedness and separatedness properties of $M_{\rm Hod}$, an argument similar to that of the previous paragraph will allow us to show
boundedness of $\Psi$ globally over $M_{\rm Hod}$ without restricting to a single fiber. 
Suppose $p_i\rightarrow q$ in $M_{\rm Hod}$
over a convergent sequence $\lambda _i\rightarrow \lambda \in \aaa ^1$, but with $|\Psi (p_i)|\rightarrow \infty$. 
Fix a preferred section $\sigma : \aaa ^1\rightarrow M_{\rm Hod}$ and we may assume that $\sigma (\lambda _i)$ is connected to $p_i$ by
a path $\gamma _i : [0,1]\rightarrow M_{\lambda _i}$. If necessary replacing $X$ by a sufficiently high genus covering, we can
view $M_{\rm Hod}$ as a family of connected normal varieties. Thus  
we can assume that the paths $\gamma _i$ converge to a path $\gamma$ connecting
$q$ to $\sigma (\lambda )$ in the fiber $M_{\lambda}$.
Fix a constant $C > {\rm sup}_t |\Psi (\gamma (t))|$, in particular also $C>|\Psi (\sigma (\lambda _i))|$, indeed
the $\Psi (\sigma (\lambda _i))$ are all the same because $\sigma$ was a preferred section.  
For large values of $i$ we have 
$$
|\Psi (\sigma (\lambda _i))|=|\Psi (\gamma _i(0))| < C < |\Psi (\gamma _i(1))| = |\Psi (p_i)|.
$$
Continuity of $\Psi$ in each fiber $M_{\lambda _i}$ shows that there are $t_i\in [0,1]$ with 
$|\Psi (\gamma _i(t_i))| = C$. Going to a subsequence we get convergence of the harmonic bundles
associated to the points $\gamma _i(t_i)$, keeping the same norm of the Hitchin map. Thus $\gamma _i(t_i)\rightarrow q'$ with
$|\Psi (q')| = C$. For a further subsequence, $t_i\rightarrow t$
and as in the previous argument, separatedness of the moduli space implies that $q' = \gamma (t)$, contradicting the choice of $C$.  
This proves that $\Psi$ is locally bounded, which in turn
implies properness of the map in the first statement of the lemma for $GL(r)$, to complete the proof.
\end{proof}

Suppose now given an injective group homomorphism between reductive groups $G\hookrightarrow H$. 
Choose compact real forms $J\subset G$ and $K\subset H$
such that the homomorphism is compatible: $J\rightarrow H$. This gives a diagram
$$
\begin{array}{ccc}
R^J_{\rm Hod}(X,x,G)&\rightarrow &R^K_{\rm Hod}(X,x,H)\\
\downarrow & & \downarrow \\
M_{\rm Hod}(X,G)&\rightarrow &M_{\rm Hod}(X,H)
\end{array}
$$
where the vertical maps are proper by the previous lemma. The upper  horizontal map is a closed embedding, indeed
we can choose $H\hookrightarrow GL(r)$ which also induces $G\hookrightarrow GL(r)$, and 
the schemes $R_{\rm Hod}(X,x,G)$ and $R_{\rm Hod}(X,x,H)$ are by construction closed subschemes
of $R_{\rm Hod}(X,x,GL(r))$ \cite{moduli}. The subsets of framings compatible with the harmonic metrics
are closed, so the upper horizontal map is an inclusion compatible with closed embeddings into
$R_{\rm Hod}(X,x,GL(r))$, hence it is a closed embedding. In particular, it is proper. This implies that
the bottom map is proper. 

\begin{corollary}
\label{finiteness}
Given a group homomorphism with finite kernel between reductive groups $G\rightarrow H$ the resulting map on moduli spaces
$$
M_{\rm Hod}(X,G)\rightarrow M_{\rm Hod}(X,H)
$$
is finite. 
\end{corollary}
\begin{proof}
If $G\rightarrow H$ is injective, the above argument shows that the map on moduli spaces is proper.
It is quasi-finite \cite{moduli} so it is finite. Then the same argument as in \cite{moduli} 
yields the same statement in the case of a group homomorphism with finite kernel between reductive groups. 
\end{proof}

\begin{corollary}
Suppose $G$ is a reductive group. Then for any  point $p\in M_{\rm Hod}(X,G)$ the limit point
$\lim _{t\rightarrow 0}t\cdot p$ exists and is unique in the fixed point set $M_{\rm H}(X,G)^{\Gm}$.
Hence we get a stratification of $M_{\rm DR}(X,G)$ just as in Proposition \ref{strata}.
\end{corollary}
\begin{proof}
Choose $G\hookrightarrow GL(r)$; apply Lemma \ref{limitlem} together with the finiteness of Corollary \ref{finiteness}
to get existence of the limit. Uniqueness follows from separatedness of the moduli space \cite{moduli}.
The stratification is defined in the same way as in \ref{strata}.
\end{proof}

The above arguments prove that the limit  points exist, however it would be good to have a geometric
construction analogous to what we did in \S \ref{gtfilt}, \S \ref{constr}. This should involve 
a principal-bundle approach to the instability flag \cite{RamananRamanathan}.

\begin{question}
How to give an explicit description of the limiting points in terms of Griffiths-transverse parabolic reductions
in the case of principal $G$-bundles?
\end{question}

We obtain the {\em oper stratification} of $M_{\rm DR}(X,G)$ just as in Proposition \ref{strata}. 
The smallest stratum consisting of $G$-opers is treated in much detail in \cite{BDopers}.

It would be good to generalize the other elements of our discussion \S\S \ref{sec-os},  \ref{sec-parabolic},  \ref{deformations} to the principal bundle case too. 
The theory of parabolic structures would hit the same complications mentionned by Seshadri in
the present conference. 
The theory of deformations should follow \cite{BiswasRamanan}.

\bibliographystyle{amsplain}

\end{document}